\documentclass[12pt, a4paper]{amsart}
\usepackage[utf8]{inputenc}
\usepackage{amsmath}
\usepackage{amssymb}
\usepackage[all]{xy}
\usepackage{tikz}

\usepackage[normalem]{ulem}
\usepackage{soul}
\usepackage{color}

\setstcolor{red}

\makeatletter
\@namedef{subjclassname@2010}{%
  \textup{2010} Mathematics Subject Classification}
\makeatother

\def \P {{\mathbb P}}

\def\leq{\leqslant}
\def\geq{\geqslant}

\theoremstyle{plain}
\newtheorem{theorem}{Theorem}
\newtheorem{proposition}{Proposition}[section]
\newtheorem{lemma}[proposition]{Lemma}

\theoremstyle{remark}

\numberwithin{equation}{section}

\addtocounter{footnote}{1}

\begin{document}

\vglue -5mm

\title[THE AVERAGE SIZE OF RAMANUJAN SUMS OVER CUBIC NUMBER FIELDS]
{THE AVERAGE SIZE OF RAMANUJAN SUMS OVER CUBIC NUMBER FIELDS}
\author{Jing Ma, Huayan Sun \& Wenguang Zhai}

\address{%
Jing Ma
\\
School of Mathematics
\\
Jilin University
\\
Changchun 130012
\\
P. R. China
}
\email{jma@jlu.edu.cn}

\address{%
Huayan Sun
\\
School of Mathematics
\\
Jilin University
\\
Changchun 130012
\\
P. R. China
}
\email{hysun19@mails.jlu.edu.cn}

\address{%
Wenguang Zhai
\\Department of Mathematics
\\China University of Mining and Technology
\\Beijing 100083
\\P. R. China
}
\email{zhaiwg@hotmail.com}

\date{\today}

\begin{abstract}
Let $\textit{K}$ be a cubic number field.  In this paper, we study the  Ramanujan sums $c_{\mathcal{J}}(\mathcal{I})$, where $\mathcal{I}$ and $\mathcal{J}$ are integral ideals in $\mathcal{O}_\textit{K}$. The asymptotic behaviour of sums of $c_{\mathcal{J}}(\mathcal{I})$ over both
$\mathcal{I}$ and $\mathcal{J}$ is investigated.
\end{abstract}


\maketitle   	

\section{Introduction}

\subsection{Ramanujan sums over the rationals}

For positive integers $m$ and $n$ the  Ramanujan sum $c_{m}(n)$ is defined as
\begin{equation}\label{classical}
c_{m}(n)=\sum_{\substack{1\leq j\leq m\\gcd(j,m)=1}}e\Big(\frac{jn}{m}\Big)=\sum_{d|gcd(m,n)}d\mu\Big(\frac{m}{d}\Big),
\end{equation}
where $e(z)=e^{2\pi i z}$ and $\mu(\cdot)$ is the M\"{o}bius function.
In 2012, Chan and Kumchev \cite{Chan-Kumchev} studied the average order of $c_{m}(n)$  with respect to both $m$ and $n$.
They proved that
\begin{equation}\label{S1}
\begin{aligned}
S_{1}(X,Y)&=\sum_{1\leq m\leq X}\sum_{1\leq n\leq Y}c_{m}(n)\\
&=Y-\frac{3}{2\pi ^{2}}X^{2}+O(XY^{1/3}\log X)+O(X^{3}Y^{-1}),
\end{aligned}
\end{equation}
for large real numbers $Y\geq X \geq 3$,
and
\begin{equation}\label{def:S_{1}(X,Y)}
\begin{aligned}
S_{1}(X,Y)
& := \begin{cases}
Y, & \text{if $\delta >2$,}
\\\noalign{\vskip 1mm}
 -\frac{3}{2\pi ^{2}}X^{2}, & \text{if $1< \delta <2$, }
\end{cases}
\end{aligned}
\end{equation}
if $Y\asymp X^{\delta}$.
\subsection{Ramanujan sums in fields}

  Let  $\textit{K}$ be a number field and $\mathcal{O}_{\textit{K}}$ denote  its ring of algebraic integers. For any nonzero integral ideal $\mathcal{I}$  in $\mathcal{O}_{\textit{K}}$, the M\"{o}bius function is defined as follows: $\mu(\mathcal{I})=0$ if there exists a prime ideal $\mathcal{P}$ such that $\mathcal{P}^{2}$ divides $\mathcal{I}$, and $\mu(\mathcal{I})=(-1)^{r}$ if $\mathcal{I}$ is a product of $r$ distinct prime ideals. For any ideal $\mathcal{I}$, the norm of $\mathcal{I}$ is denoted by $\textit{N}(\mathcal{I})$.
For nonzero integral ideals $\mathcal{I}$ and $\mathcal{J}$,
the Ramanujan sum  in fields is defined by
\begin{equation}\label{fields}
c_{\mathcal{J}}(\mathcal{I})=\sum_{\substack{\mathcal{M}\in \mathcal{O}_{\textit{K}}\\\mathcal{M}|\mathcal{I},\mathcal{M}|\mathcal{J}}}\textit{N}(\mathcal{M})\mu\Big(\frac{\mathcal{J}}{\mathcal{M}}\Big),
\end{equation}
which is an analogue of \eqref{classical}.

For each $n\geq 1$, let  $a_{\textit{K}}(n)$  denote the number of integral ideals in $\mathcal{O}_{\textit{K}}$ of norm $ n$.
Then
\begin{equation}\label{dd}
\sum_{n\leq x}a_{\textit{K}}(n)=\rho_{\textit{K}} x +P_{\textit{K}}(x),\qquad P_{\textit{K}}(x)=O(x^{\frac{\mathbf{d}-1}{\mathbf{d}+1}}),
\end{equation}
where $\rho_{\textit{K}}$ is a constant depending only on  the field $\textit{K}$ and $\mathbf{d}$ is the  degree of the  field extension $\textit{K}/\mathbb{Q}$. This is a classical result of Landau  (see \cite{Landau}).

Let $X \geq 3$ and $Y \geq 3$ be two large real numbers. Define
\begin{equation}
S_{\textit{K}}(X,Y):=\sum_{1\leq \textit{N}(\mathcal{J})\leq X}\sum_{1\leq \textit{N}(\mathcal{I})\leq Y}c_{\mathcal{J}}(\mathcal{I}),
\end{equation}
which is an analogue of \eqref{S1}. 

When $\textit{K}$ is a  quadratic number field, some authors studied the asymptotic behaviour of $S_{\textit{K}}(X,Y)$ (see \cite{Nowak}, \cite{Zhai}, \cite{Zhai-2}).
In \cite{Nowak},  Nowak proved
\begin{equation}\label{1.8}
S_{\textit{K}}(X,Y)\sim \rho_{\textit{K}} Y
\end{equation}
 provided that $Y>X^{\delta}$ for some $\delta>\frac{1973}{820}$.
In \cite{Zhai}, Zhai improved Nowak' results and  proved that \eqref{1.8} holds provided that $Y>X^{\delta}$  for some $\delta>\frac{79}{34}$.
Recently Zhai \cite{Zhai-2} proved that \eqref{1.8} holds for $Y>X^{2+\varepsilon}$.

In this paper, we consider the asymptotic behaviour of $S_{\textit{K}}(X,Y)$ for   a cubic field $\textit{K}$.  We shall prove the following
results.

\begin{theorem}\label{Thm 1}
Let  $\textit{K}$ be a  cubic number field. Suppose that $Y\geq X\geq 3$ are  large real numbers.
 Then
\begin{equation}
S_{\textit{K}}(X,Y)=\rho_{\textit{K}} Y+O(X^{\frac{8}{5}}Y^{\frac{2}{5}+\varepsilon}+X^{\frac{11}{8}}Y^{\frac{1}{2}+\varepsilon})
\end{equation}
provided that $Y>X^{11/4}$.
\end{theorem}

\begin{theorem}\label{Thm 2}
Let $\textit{K}$ be a  cubic number field.  Suppose that $T\geq X\geq 3$ are two large real numbers such that $T\geq 10X$.  Then we have
\begin{equation*}
\int_{T}^{2T}|\mathfrak{R}_{\textit{K}}(X,Y)|^{2}dY=
c(X)\int_{T}^{2T}Y^{\frac{2}{3}}dY+
O(X^{\frac{31}{9}}T^{\frac{14}{9}+\varepsilon}+X^{\frac{26}{9}}T^{\frac{29}{18}+\varepsilon}),
\end{equation*}
where $$
\mathfrak{R}_{\textit{K}}(X,Y):=S_{\textit{K}}(X,Y)-\rho_{\textit{K}} Y
$$ and  $c(X)$ is defined by \eqref{cx}.
\end{theorem}

\noindent{\bf Remark}.
From  \eqref{c(x)} we can see that $c(X)\ll  X^{\frac{7}{3}+\varepsilon} $. From this estimate we get from Theorem 2 that the asymptotic formula
\eqref{1.8} holds on average provided that $Y>X^{\frac{7}{3}+\varepsilon}$.

\noindent{\bf Notation}.
Let $[x]$ denote the greatest integer  less or equal to $x$.
The notation $U\ll V$ means that there exists  a constant $C>0$ such that  $|U|\leqslant CV$, which is equivalent to $U=O(V)$.  The notations  $U\gg V$ (which implies $U\geqslant 0$ and $V\geqslant 0$), $U\asymp V$ (which means  that we have both $U\ll V$  and $U\gg V$ ) are defined similarly.
Let $\zeta(s)$ denote the Riemann zeta-function and $\tau_{r}(n)$ the number of  ways $n$   factorized into $r$ factors. In particular, $\tau_{2}(n)=\tau (n)$ is the Dirichlet divisor function.  At last, let $z_{n}~(n\geq 1)$ denote a series of  complex numbers.  We set
\begin{equation}
\Big|\sum_{N<n\leq 2N}z_{n}\Big|^{*}:=\max_{N\leq N_{1}<N_{2}\leq 2N}\Big|\sum_{N_{1}<n\leq N_{2}}z_{n}\Big|.
\end{equation}

\section{some lemmas}
In this section, we will make preparation for the proof of our theorems.
 From now on, we always supposed that $\textit{K}$ is a cubic  number field.  The Dedekind zeta-function of $\textit{K}$ is defined by
\begin{equation}
\zeta_{\textit{K}}(s):=\sum_{\substack{\mathcal{I}\in \mathcal{O}_{\textit{K}}\\\mathcal{I}\neq 0}}\frac{1}{\textit{N}^{s}(\mathcal{I})}\quad(\Re s >1).
\end{equation}
Then
\begin{equation}
\zeta_{\textit{K}}(s)=\sum_{n=1}^{\infty}\frac{a_{\textit{K}}(n)}{n^{s}} \quad(\Re s >1),
\end{equation}
where $a_{\textit{K}}(n)$  is  the number of integral ideals in $\mathcal{O}_{\textit{K}}$ of norm $ n$.

The function $\mu_{\textit{K}}(n)$ is defined by
$$
\frac{1}{\zeta_{\textit{K}}(s)}:=\sum_{n=1}^{\infty}\frac{\mu_{\textit{K}}(n)}{n^{s}}   \quad (\Re s>1).
$$
Define $$
\textit{M}_{\textit{K}}(x):=\sum_{n\leq x}\mu_{\textit{K}}(n).
$$
Then there is a trivial bound
\begin{equation}\label{M(x)}
\textit{M}_{\textit{K}}(x)\ll x.
\end{equation}

We collect the algebraic properties of cubic number fields as the following Lemma.
\begin{lemma}\label{lemma 1}
Let $\textit{K}$ be a cubic number field over $\mathbb{Q}$ and $D=df^{2}$ ($d$ squarefree) its discriminant; then
\begin{itemize}
\item[(a)]
$\textit{K}/\mathbb{Q}$ is a normal extension if and only if $D=f^{2}$. In this case $$
\zeta_{\textit{K}}(s)=\zeta(s)L(s,\chi _{1})\overline{L(s,\chi _{1})},
$$
where $\zeta (s)$ is the Riemann zeta-function and $L(s,\chi _{1})$ is an ordinary Dirichlet series (over $\mathbb{Q}$) corresponding to a primitive character $\chi _{1}$ modulo $f$.
\item[(b)]
If $\textit{K}/\mathbb{Q}$ is not  a normal extension, then $d\neq 1$ and $$
\zeta_{\textit{K}}(s)=\zeta(s)L(s,\chi _{2}),
$$
where $L(s,\chi _{2})$ is a Dirichlet $L$-function over the quadratic field $F =\mathbb{Q}(\sqrt{d})$:
$$
L(s,\chi _{2})=\sum_{\varrho}\chi _{2}(\varrho)N_{F}(\varrho)^{-s}, \quad (\Re s  >1).
$$
Here the summation  is taken over all ideals $\varrho\neq 0$ in $F$ and $N_{F}$ denotes the (absolute) ideal norm in $F$.
\end{itemize}
\end{lemma}

\begin{proof}

This is Lemma 1 in \cite{Muller}.
\end{proof}

\noindent{\bf Remark 1}.
To describe the character $\chi _{2}$, let $H$ be the ideal group in $F$ according to which the normal extension $\textit{K}(\sqrt{d})$  is the class field. Then $H$ divides the set  $A^{f}$ of all ideals $\varrho\subseteq F$ with $(\varrho,f)=1$ into three classes $A^{f}=H\cup C\cup C^{'}$, and ($\omega=e^{2\pi i/3}$)
\begin{equation*}
\begin{aligned}
\chi_{2}(\varrho)&=\begin{cases}
1, \qquad \varrho\in H,
\\\noalign{\vskip 1mm}
\omega,  \qquad \varrho\in C,
\\\noalign{\vskip 1mm}
\overline{\omega},  \qquad \varrho\in C^{'},
\\\noalign{\vskip 1mm}
0, \qquad (\varrho,f) \neq 1.
\end{cases}
\end{aligned}
\end{equation*}
The substitution $\gamma=(\sqrt{d} \mapsto -\sqrt{d})$ in $F$ maps $C$ onto $C^{'}$.

\noindent{\bf Remark 2}.
The factorization of $\zeta_{\textit{K}}(s)$ in Lemma \ref{lemma 1}  gives
\begin{equation}\label{a-1b}
a_{\textit{K}}(n)=\sum_{m|n}b(m),
\end{equation}
where in the case of a normal extension $b(m)=\sum_{xy=m}\chi_{1}(x)\overline{\chi_{1}(y)}$ ($\chi_{1}$ is the primitive character modulo $f$). Otherwise $b(m)$ is equal to the number of ideals $\varrho\in H$ with $N_{F}(\varrho)=m$ minus two times the number of ideals $\varrho\in C$ with $N_{F}(\varrho)=m$. In both cases, $|b(m)|\ll m^{\varepsilon}$.

\begin{lemma}\label{lemma 2}
Let $\textit{K}$ be an algebraic number field of degree $\mathbf{d}$, then
\begin{equation}
a_{\textit{K}}(n)\ll (\tau (n))^{\mathbf{d}-1},
\end{equation}
where $\tau (n)$ is the Dirichlet divisor function and $\mathbf{d}=[K:\mathbb{Q}]$.
\end{lemma}

\begin{proof}

This is (68) in \cite{Chan-Nara}.
\end{proof}

\noindent{\bf Corollary}.
Let $\textit{K}$ be a cubic field, then we have
\begin{equation}
a_{\textit{K}}(n)\ll \tau ^{2}(n).
\end{equation}

\begin{lemma}\label{lemma 3}
Suppose $1\ll N\ll Y$, then we have
\begin{equation}
P_{\textit{K}}(Y)=\frac{Y^{1/3}}{\sqrt{3}\pi}\sum_{n\leq N}\frac{a_{\textit{K}}(n)}{n^{{2}/{3}}}\cos (6\pi(nY)^{{1}/{3}})+O(Y^{{2}/{3}+\varepsilon}N^{-{1}/{3}}),
\end{equation}
where the $O$-constant depends on $\varepsilon$.
\end{lemma}

\begin{proof}
This is a special case of Proposition 3.2 of Friedlander and Iwaniec \cite{Friedlander-Iwaniec}.
\end{proof}

\begin{lemma}\label{lemma 3+}
Let $T\geq 10$ be a large parameter and $y$ a real number such that $T^{\varepsilon}\ll y\ll T$. Define for any $T\leq Y \leq 2T$ that
\begin{equation*}
\begin{split}
&P_{1}(Y)=P_{1}(Y;y):=\frac{Y^{1/3}}{\sqrt{3}\pi}\sum_{n\leq y}\frac{a_{\textit{K}}(n)}{n^{{2}/{3}}}\cos (6\pi(nY)^{{1}/{3}}),\\
&P_{2}(Y)=P_{2}(Y;y):=P_{\textit{K}}(Y)-P_{1}(Y).
\end{split}
\end{equation*}
Then we have
\begin{equation}\label{zzzzzz}
\int_{T}^{2T}|P_{2}(Y)|^{2}dY
\ll
{T^{5/3+\varepsilon}}{y^{-1/3}}\quad(y\ll T^{1/3}).
\end{equation}
\end{lemma}

\begin{proof}
We can prove that the estimate
\begin{equation}\label{zeta -K-7/12}
\int_{1}^{T}|\zeta_{\textit{K}}(7/12+it)|^{2}dt\ll T^{1+\varepsilon}
\end{equation}
holds.

If $\textit{K}/\mathbb{Q}$ is a normal extension, then by Lemma \ref{lemma 1} we have $\zeta_{\textit{K}}(s)=\zeta(s)L(s,\chi _{1})\overline{L(s,\chi _{1})}$.
From  Theorem 8.4  in \cite{Ivi-book} we get that
\begin{equation}\label{zeta7/12}
\int_{1}^{T}|\zeta(7/12+it)|^{6}dt\ll T^{1+\varepsilon}.
\end{equation}
The proof of  Theorem 8.4 in \cite{Ivi-book} can be applied directly to $L(s,\chi_{1})$ to derive
\begin{equation}\label{L7/12}
\int_{1}^{T}|L(7/12+it,\chi_{1})|^{6}dt\ll T^{1+\varepsilon}.
\end{equation}
From \eqref{zeta7/12},  \eqref{L7/12} and H\"{o}lder's inequality we get
\begin{equation*}
\begin{split}
&\int_{1}^{T}|\zeta_{\textit{K}}(7/12+it)|^{2}dt\\
&=\int_{1}^{T}|\zeta(7/12+it)|^{2}|L(7/12+it,\chi_{1})|^{4}dt\\
&\ll \Big(\int_{1}^{T}|\zeta(7/12+it)|^{6}dt\Big)^{1/3}\Big(\int_{1}^{T}|L(7/12+it,\chi_{1})|^{6}dt\Big)^{2/3}\\
&\ll T^{1+\varepsilon}.
\end{split}
\end{equation*}

Now suppose that $\textit{K}/\mathbb{Q}$ is not  a normal extension, then $\zeta_{\textit{K}}(s)=\zeta(s)L(s,\chi_{2})$ from Lemma \ref{lemma 1}.  We
know that $L(s,\chi_{2})$ is an  automorphic $L$-function of degree 2 corresponding to a cusp  form $F$
 over $SL_{2}(\mathbb{Z})$ (see, for example, Fomenko \cite{Fomenko}). So from \cite[Lemma 12]{Zhai-Cao}, which is originally proved in   \cite{Ivi},  we have
\begin{equation}\label{Ivi1.6}
\int_{1}^{T}|L(7/12+it,\chi_{2})|^{3}dt\ll T^{1+\varepsilon}.
\end{equation}
By \eqref{zeta7/12}, \eqref{Ivi1.6}  and H\"{o}lder's inequality we get
\begin{equation*}
\begin{split}
&\int_{1}^{T}|\zeta_{\textit{K}}(7/12+it)|^{2}dt\\
&=\int_{1}^{T}|\zeta(7/12+it)|^{2}|L(7/12+it,\chi_{2})|^{2}dt\\
&\ll \Big(\int_{1}^{T}|\zeta(7/12+it)|^{6}dt\Big)^{1/3}\Big(\int_{1}^{T}|L(7/12+it,\chi_{2})|^{3}dt\Big)^{2/3}\\
&\ll T^{1+\varepsilon}.
\end{split}
\end{equation*}

Now we give a short proof  of \eqref{zzzzzz}. For simplicity, we follow the proof of Theorem 1 in \cite{Zhai-Cao}. Take $d=3$, $a(n)=a_{\textit{K}}(n)$, $N=[T^{5-\varepsilon}]$ and $M=[T^{2/3}]$. From \eqref{zeta -K-7/12} we can take $\sigma^{\ast}=7/12$. As in the proof of Theorem 1 in \cite{Zhai-Cao}, we can write
\begin{equation*}
P_{2}(Y)=R_{1}^{\ast}(Y;y)+\sum_{j=2}^{7}R_{j}(Y),
\end{equation*}
where
\begin{equation*}
R_{1}^{\ast}(Y;y):=\frac{Y^{1/3}}{\sqrt{3}\pi}\sum_{y<n\leq M}\frac{d_{3}(n)}{n^{2/3}}\cos (6\pi(nY)^{1/3})
\end{equation*}
and $R_{j}(Y)$ (j = 2, 3, 4, 5, 6, 7) were defined in page 2129 of \cite{Zhai-Cao}. Similar to the formula (8.11) of \cite{Zhai-Cao}, we have the estimate (noting that $y\ll T^{1/3}$)
\begin{equation*}
\begin{split}
&\int_{T}^{2T}(R_{1}^{\ast}(x;y)+R_{2}(x))^{2}dx\\
&\ll \sum_{y<n\leq M}\frac{d_{3}^{2}(n)}{n^{4/3}}\int_{T}^{2T}x^{2/3}dx+T^{5/3+\varepsilon}M^{-1/6}+T^{4/3+\varepsilon}M^{1/3}\\
&\ll T^{5/3+\varepsilon}y^{-1/3}+T^{14/9+\varepsilon}\ll  T^{5/3+\varepsilon}y^{-1/3},
\end{split}
\end{equation*}
which combining (8.17) of \cite{Zhai-Cao} gives \eqref{zzzzzz}.

\end{proof}

Next, we consider the following exponential sums
\begin{equation}
S_{0}=\sum_{H<h\leq 2H}\sum_{N<n\leq 2N}a(h,n)\sum_{{M<m\leq 2M}}b(m)e\Big(U\frac{h^{\beta}n^{\gamma}m^{\alpha}}{H^{\beta}N^{\gamma}M^{\alpha}}\Big)
\end{equation}
and
\begin{equation}
S_{1}=\sum_{H<h\leq 2H}\sum_{N<n\leq 2N}a(h,n)\Big|\sum_{{M<m\leq 2M}}e\Big(U\frac{h^{\beta}n^{\gamma}m^{\alpha}}{H^{\beta}N^{\gamma}M^{\alpha}}\Big)\Big|^{*},
\end{equation}
where $H, N, M$ are positive integers, $U$ is a real number greater than one, $a(h,n)$ and $b(m)$ is a complex number of modulus at most one; moreover, $\alpha, \beta ,\gamma$ are fixed real numbers such that $\alpha(\alpha-1)\beta\gamma\neq 0$.

 Then we have

\begin{lemma}\label{lemma 4}
\begin{equation}
S_{0}\ll(HNM)^{1+\varepsilon}\Big(\Big(\frac{U}{HNM^{2}}\Big)^{1/4}+\frac{1}{(HN)^{1/4}}+\frac{1}{M^{1/2}}+\frac{1}{U^{1/2}}\Big),
\end{equation}
and
\begin{equation}
S_{1}\ll(HNM)^{1+\varepsilon}\Big(\Big(\frac{U}{HNM^{2}}\Big)^{1/4}+\frac{1}{M^{1/2}}+\frac{1}{U}\Big).
\end{equation}
\end{lemma}

\begin{proof}

This is contained in \cite{Robert and Sargos}.
\end{proof}

\begin{lemma}\label{lemma 5}
Suppose that
$$
L(H)=\sum_{i=1}^{m}A_{i}H^{a_{i}}+\sum_{j=1}^{n}B_{j}H^{-b_{j}},
$$
where $A_{i},B_{j},a_{i},$ and $b_{j}$ are positive. Assume that $H_{1}\leq H_{2}$. Then there is some $H$ with $H_{1}\leq H \leq H_{2}$ and
$$
L(H)\ll\sum_{i=1}^{m}\sum_{j=1}^{n}(A_{i}^{b_{j}}B_{j}^{a_{i}})^{{1}/{(a_{i}+b_{j})}}+\sum_{i=1}^{m}A_{i}H_{1}^{a_{i}}+\sum_{j=1}^{n}B_{j}H_{2}^{-b_{j}}.
$$
The implied constants depend only on $m$ and $n$.
\end{lemma}

\begin{proof}

See Lemma 2.4 in \cite{GrahamKolesnik1991}.
\end{proof}

\begin{lemma}\label{lemma 6}
Let $l\geq2$ and $q\geq 1$ be two fixed integers. Then we have
\begin{equation}
\sum_{n\leq x}\tau_{l}^{q}(n)\ll x(\log x)^{l^{q}-1}.
\end{equation}
\end{lemma}

\begin{proof}

See Lemma 2.4 in \cite{Zhai}.
\end{proof}

\begin{lemma}\label{lemma 7}
Let $T \geq 2$ be a real number. Then we have
\begin{equation}
\sum_{\substack{m,n\leq T\\ m\neq n}}\frac{\tau_{4}^{2}(m)\tau_{4}^{2}(n)}{(mn)^{\frac{2}{3}}|\sqrt[3]{m}-\sqrt[3]{n}|}\ll T^{\frac{1}{3}+\varepsilon}.
\end{equation}
\end{lemma}
\begin{proof}

Firstly, we write
$$
\sum_{\substack{m,n\leq T\\ m\neq n}}\frac{\tau_{4}^{2}(m)\tau_{4}^{2}(n)}{(mn)^{\frac{2}{3}}|\sqrt[3]{m}-\sqrt[3]{n}|}=S_{1}+S_{2},
$$
where
$$
S_{1}=\sum_{\substack{m,n\leq T\\ |\sqrt[3]{m}-\sqrt[3]{n}|\geq (mn)^{1/6}/10}}\frac{\tau_{4}^{2}(m)\tau_{4}^{2}(n)}{(mn)^{\frac{2}{3}}|\sqrt[3]{m}-\sqrt[3]{n}|},
$$
$$
S_{2}=\sum_{\substack{m,n\leq T\\ 0<|\sqrt[3]{m}-\sqrt[3]{n}|< (mn)^{1/6}/10}}\frac{\tau_{4}^{2}(m)\tau_{4}^{2}(n)}{(mn)^{\frac{2}{3}}|\sqrt[3]{m}-\sqrt[3]{n}|}.
$$
Applying Lemma \ref{lemma 6} with $l=4$ and $q=2$, we have
$$
S_{1}\ll \sum_{m,n\leq T}\frac{\tau_{4}^{2}(m)\tau_{4}^{2}(n)}{(mn)^{\frac{5}{6}}}\ll T^{\frac{1}{3}+\varepsilon},
$$
where we used partial summation.

Secondly,  $0<|\sqrt[3]{m}-\sqrt[3]{n}|< {(mn)^{1/6}}/{10}$ implies that $m\asymp n$. And from the Lagrange theorem we have $|\sqrt[3]{m}-\sqrt[3]{n}|\asymp (mn)^{-1/3}|m-n|$.
By the formula $ab\leq (a^{2}+b^{2})/2$ and Lemma \ref{lemma 6} with $l=4$ and $q=4$ we get that
\begin{equation*}
\begin{aligned}
S_{2}&\ll\sum_{m\asymp n\leq T}\frac{\tau_{4}^{2}(m)\tau_{4}^{2}(n)}{(mn)^{{1}/{3}}|m-n|}\\
     &\ll\sum_{m\asymp n\leq T}\Big(\frac{\tau_{4}^{4}(m)}{m^{{2}/{3}}}+\frac{\tau_{4}^{4}(n)}{n^{{2}/{3}}}\Big)\frac{1}{|m-n|}\\
     &\ll\sum_{m\leq T}\frac{\tau_{4}^{4}(m)}{m^{{2}/{3}}}\sum_{m\asymp n}\frac{1}{|m-n|}\ll T^{\frac{1}{3}+\varepsilon}.
\end{aligned}
\end{equation*}

\end{proof}

\section{The proof of Theorem 1}
We begin  the proof  with formula (2.3) in \cite{Nowak}, which reads
\begin{equation}\label{3.1}
\begin{aligned}
S_{\textit{K}}(X,Y)&=\rho_{\textit{K}} Y+ \sum_{\substack{\mathcal{M},\mathcal{L}\in \mathcal{O}_{\textit{K}}\\1\leq N(\mathcal{M}\mathcal{L})\leq X}}N(\mathcal{M})\mu(\mathcal{L})P_{\textit{K}}\Big(\frac{Y}{N(\mathcal{M})}\Big)\\
      &=\rho_{\textit{K}} Y+\sum_{\substack{\mathcal{M},\mathcal{L}\in \mathcal{O}_{\textit{K}}\\1\leq N(\mathcal{M})N(\mathcal{L})\leq X}}N(\mathcal{M})\mu(\mathcal{L})P_{\textit{K}}\Big(\frac{Y}{N(\mathcal{M})}\Big).
\end{aligned}
\end{equation}
Let $\mathfrak{R}=\mathfrak{R}_{\textit{K}}(X,Y)$ denote the last sum in \eqref{3.1}.  We have
\begin{equation}\label{2zong}
\begin{aligned}
\mathfrak{R}&=\sum_{1\leq ml \leq X}ma_{\textit{K}}(m)\mu_{K}(l)P_{\textit{K}}\Big(\frac{Y}{m}\Big)\\
            &=\sum_{1\leq l \leq X}\mu_{K}(l)\sum_{1\leq m \leq X/l}ma_{\textit{K}}(m)P_{\textit{K}}\Big(\frac{Y}{m}\Big)\\
            &=\mathfrak{R_{1}^{\dag}}+\mathfrak{R_{2}^{\dag}},
\end{aligned}
\end{equation}
where
\begin{equation*}
\begin{aligned}
&\mathfrak{R_{1}^{\dag}}:=\sum_{1\leq l \leq X^{1-\varepsilon}}\mu_{K}(l)\sum_{1\leq m \leq X/l}ma_{\textit{K}}(m)P_{\textit{K}}\Big(\frac{Y}{m}\Big),\\
&\mathfrak{R_{2}^{\dag}}:=\sum_{X^{1-\varepsilon}< l \leq X}\mu_{K}(l)\sum_{1\leq m \leq X/l}ma_{\textit{K}}(m)P_{\textit{K}}\Big(\frac{Y}{m}\Big).
\end{aligned}
\end{equation*}

Firstly, we bound $\mathfrak{R_{2}^{\dag}}$.
M\"{u}ller \cite{Muller} proved that $P_{\textit{K}}(x)=O(x^{\frac{43}{96}+\varepsilon})$.
So we can easily derived that
\begin{equation}\label{R2}
\mathfrak{R_{2}^{\dag}}\ll XY^{43/96+\varepsilon}.
\end{equation}

Secondly, we consider $\mathfrak{R_{1}^{\dag}}$.
We can write
\begin{equation}\label{R_1+}
\mathfrak{R_{1}^{\dag}}:=\sum_{1\leq l \leq X^{1-\varepsilon}}\mu_{K}(l)\mathfrak{R_{1}}(X_{l},Y),
\end{equation}
where
\begin{equation}\label{R1-1}
\mathfrak{R_{1}}(X_{l},Y)=\sum_{1\leq m \leq X_{l}}ma_{\textit{K}}(m)P_{\textit{K}}\Big(\frac{Y}{m}\Big),\qquad
X_{l}=X/l.
\end{equation}
Using \eqref{a-1b},
we can write
\begin{equation}\label{R1-2}
\mathfrak{R_{1}}(X_{l},Y)=\sum_{1\leq m_{1}m_{2}\leq X_{l}}m_{1}m_{2}b(m_{2})P_{\textit{K}}\Big(\frac{Y}{m_{1}m_{2}}\Big).
\end{equation}
By a splitting argument, $\mathfrak{R_{1}}(X_{l},Y)$ can be written as a  sum of the following
\begin{equation}\label{R1-3}
R(M_{1},M_{2}):=\sum_{\substack{1\leq m_{1}m_{2}\leq X_{l}\\M_{j}<m_{j}\leq 2M_{j}(j=1,2)}}m_{1}m_{2}b(m_{2})P_{\textit{K}}\Big(\frac{Y}{m_{1}m_{2}}\Big).
\end{equation}
Suppose that $ y\ll Y/ M_{1}M_{2}$ is a parameter to be determined. By Lemma \ref{lemma 3}, we have
\begin{equation*}
\begin{aligned}
R(M_{1},M_{2})&=\frac{Y^{\frac{1}{3}}}{\sqrt{3}\pi}\sum_{\substack{1\leq m_{1}m_{2}\leq X_{l}\\M_{j}<m_{j}\leq 2M_{j}(j=1,2)}}(m_{1}m_{2})^{\frac{2}{3}}b(m_{2})\sum_{n\leq y}\frac{a_{\textit{K}}(n)}{n^{2/3}}\cos\Big(6\pi\sqrt[3]{\frac{nY}{m_{1}m_{2}}}\Big)\\
&\qquad \qquad+O((M_{1}M_{2})^{4/3}Y^{2/3+\varepsilon}y^{-1/3}).
\end{aligned}
\end{equation*}
By a splitting argument to the sum over $n$ we get that
\begin{equation}\label{R(m1,m2)}
\begin{aligned}
R(M_{1},M_{2})&\ll Y^{\frac{1}{3}}(M_{1}M_{2})^{\frac{2}{3}+\varepsilon}N^{-\frac{2}{3}+\varepsilon}|R^{*}(M_{1},M_{2},N)|\\
&+O((M_{1}M_{2})^{4/3}Y^{2/3+\varepsilon}y^{-1/3})
\end{aligned}
\end{equation}
for some $1\ll N\ll y$, where
\begin{equation*}
\begin{aligned}
R^{*}(M_{1},M_{2},N)=\sum_{\substack{1\leq m_{1}m_{2}\leq X_{l}\\M_{j}<m_{j}\leq 2M_{j}(j=1,2)}}\Big(\frac{m_{1}}{M_{1}}\Big)^{\frac{2}{3}}\Big(\frac{m_{2}}{M_{2}}\Big)^{\frac{2}{3}}\frac{b(m_{2})}{M_{2}^{\varepsilon}}\sum_{N<n \leq 2N}c(n)e\Big(6\pi\sqrt[3]{\frac{nY}{m_{1}m_{2}}}\Big)
\end{aligned}
\end{equation*}
with
$$
c(n)=\frac{a_{\textit{K}}(n)}{N^{\varepsilon}}\Big(\frac{N}{n}\Big)^{\frac{2}{3}}.
$$

Now, we give our first estimate for the sum $R^{*}(M_{1},M_{2},N)$.
Obviously, we have
\begin{equation}\label{111}
R^{*}(M_{1},M_{2},N)\ll R^{\dag}(M_{1},M_{2},N),
\end{equation}
where
$$
R^{\dag}(M_{1},M_{2},N)=\sum_{M_{2}<m_{2}\leq 2M_{2}}\sum_{N<n\leq 2N}\Big|\sum_{M_{1}<m_{1}\leq 2M_{1}}e\Big(6\pi\sqrt[3]{\frac{nY}{m_{1}m_{2}}}\Big)\Big|^{*}.
$$
By taking $(H,N,M)=(M_{2},N,M_{1})$ and $U=\sqrt[3]{NY}/\sqrt[3]{M_{1}M_{2}}$ in Lemma \ref{lemma 4}, we get that
$$R^{\dag}(M_{1},M_{2},N)Y^{-\varepsilon}\ll N^{\frac{5}{6}}Y^{\frac{1}{12}}{M_{1}}^{\frac{5}{12}}M_{2}^{\frac{2}{3}}+N{M_{1}}^{\frac{1}{2}}M_{2}+N^{\frac{2}{3}}Y^{-\frac{1}{3}}(M_{1}M_{2})^{\frac{4}{3}},
$$
which combines \eqref{111} gives
\begin{equation}\label{1111}
\begin{aligned}
&R^{*}(M_{1},M_{2},N)Y^{-\varepsilon}\\
&\ll N^{\frac{5}{6}}Y^{\frac{1}{12}}{M_{1}}^{\frac{5}{12}}M_{2}^{\frac{2}{3}}+N{M_{1}}^{\frac{1}{2}}M_{2}+N^{\frac{2}{3}}Y^{-\frac{1}{3}}(M_{1}M_{2})^{\frac{4}{3}},\\
&=N^{\frac{5}{6}}Y^{\frac{1}{12}}(M_{1}M_{2})^{\frac{5}{12}}M_{2}^{\frac{1}{4}}+N(M_{1}M_{2})^{\frac{1}{2}}M_{2}^{\frac{1}{2}}+N^{\frac{2}{3}}Y^{-\frac{1}{3}}(M_{1}M_{2})^{\frac{4}{3}}.
\end{aligned}
\end{equation}

Next, we give another estimate for $R^{*}(M_{1}, M_{2}, N)$.
Clearly we have
\begin{equation}\label{?}
R^{*}(M_{1}, M_{2}, N)\ll R^{\ddag}(M_{1}, M_{2}, N),
\end{equation}
where
\begin{equation*}
 R^{\ddag}(M_{1}, M_{2}, N)=\sum_{M_{1}<m_{1}\leq 2M_{1}}\sum_{N<n\leq 2N}\sum_{M_{2}<m_{2}\leq 2M_{2}}e\Big(6\pi\sqrt[3]{\frac{nY}{m_{1}m_{2}}}\Big).
\end{equation*}
By taking $(H, N, M)=(M_{1}, N, M_{2})$ and $U=\sqrt[3]{NY}/\sqrt[3]{M_{1}M_{2}}$ in Lemma \ref{lemma 4}, we get that
\begin{equation*}
\begin{split}
R^{\ddag}(M_{1},M_{2},N)Y^{-\varepsilon}
&\ll
N^{\frac{5}{6}}Y^{\frac{1}{12}}(M_{1}M_{2})^{\frac{5}{12}}M_{1}^{\frac{1}{4}}+N^{\frac{3}{4}}(M_{1}M_{2})^{\frac{3}{4}}M_{2}^{\frac{1}{4}}\\
&\qquad \qquad+N(M_{1}M_{2})^{\frac{1}{2}}M_{1}^{\frac{1}{2}}+N^{\frac{5}{6}}Y^{-\frac{1}{6}}(M_{1}M_{2})^{\frac{7}{6}}\\
&\ll
N^{\frac{5}{6}}Y^{\frac{1}{12}}(M_{1}M_{2})^{\frac{5}{12}}M_{1}^{\frac{1}{4}}+N^{\frac{3}{4}}(M_{1}M_{2})^{\frac{3}{4}}(M_{1}M_{2})^{\frac{1}{4}}
\\
&\qquad \qquad+N(M_{1}M_{2})^{\frac{1}{2}}M_{1}^{\frac{1}{2}}+N^{\frac{5}{6}}Y^{-\frac{1}{6}}(M_{1}M_{2})^{\frac{7}{6}}\\
&\ll
N^{\frac{5}{6}}Y^{\frac{1}{12}}(M_{1}M_{2})^{\frac{5}{12}}M_{1}^{\frac{1}{4}}+N^{\frac{3}{4}}(M_{1}M_{2})
\\
&\qquad \qquad+N(M_{1}M_{2})^{\frac{1}{2}}M_{1}^{\frac{1}{2}}+N^{\frac{5}{6}}Y^{-\frac{1}{6}}(M_{1}M_{2})^{\frac{7}{6}}.
\end{split}
\end{equation*}
So
\begin{equation}\label{2222}
\begin{split}
R^{*}(M_{1},M_{2},N)Y^{-\varepsilon}
&\ll
N^{\frac{5}{6}}Y^{\frac{1}{12}}(M_{1}M_{2})^{\frac{5}{12}}M_{1}^{\frac{1}{4}}+N(M_{1}M_{2})^{\frac{1}{2}}M_{1}^{\frac{1}{2}}
\\
&\qquad \qquad+N^{\frac{3}{4}}(M_{1}M_{2})+N^{\frac{5}{6}}Y^{-\frac{1}{6}}(M_{1}M_{2})^{\frac{7}{6}}.
\end{split}
\end{equation}

From \eqref{1111} and \eqref{2222}, we get
\begin{equation*}
\begin{split}
R^{*}(M_{1},M_{2},N)Y^{-\varepsilon}&\ll J_{1}+J_{2}+J_{3}+J_{4}+
N^{\frac{5}{6}}Y^{-\frac{1}{6}}(M_{1}M_{2})^{\frac{7}{6}}\\
&\qquad +N^{\frac{3}{4}}(M_{1}M_{2})
+N^{\frac{2}{3}}Y^{-\frac{1}{3}}(M_{1}M_{2})^{\frac{4}{3}},
\end{split}
\end{equation*}
where
\begin{equation*}
\begin{aligned}
&J_{1}=\min\Big(N(M_{1}M_{2})^{\frac{1}{2}}M_{2}^{\frac{1}{2}},N^{\frac{5}{6}}Y^{\frac{1}{12}}(M_{1}M_{2})^{\frac{5}{12}}M_{1}^{\frac{1}{4}}\Big),\\
&J_{2}=\min\Big(N^{\frac{5}{6}}Y^{\frac{1}{12}}(M_{1}M_{2})^{\frac{5}{12}}M_{2}^{\frac{1}{4}},N^{\frac{5}{6}}Y^{\frac{1}{12}}(M_{1}M_{2})^{\frac{5}{12}}M_{1}^{\frac{1}{4}}\Big),\\
&J_{3}=\min\Big(N(M_{1}M_{2})^{\frac{1}{2}}M_{2}^{\frac{1}{2}},N(M_{1}M_{2})^{\frac{1}{2}}M_{1}^{\frac{1}{2}}\Big),\\
&J_{4}=\min\Big(N^{\frac{5}{6}}Y^{\frac{1}{12}}(M_{1}M_{2})^{\frac{5}{12}}M_{2}^{\frac{1}{4}},N(M_{1}M_{2})^{\frac{1}{2}}M_{1}^{\frac{1}{2}}\Big).
\end{aligned}
\end{equation*}
Noticing the fact that
$
\min(X_{1},\ldots,X_{k}) \leq X_{1}^{a_{1}} \ldots   X_{k}^{a_{k}},
$
where
$X_{1}, \ldots,X_{k}>0$, $a_{1},\ldots,a_{k}\geq 0$ satisfying  $a_{1}+\ldots+a_{k}=1$,
we have
\begin{equation*}
\begin{aligned}
J_{1}&\leq  \Big(N(M_{1}M_{2})^{\frac{1}{2}}M_{2}^{\frac{1}{2}}\Big)^{\frac{1}{3}}\Big(N^{\frac{5}{6}}Y^{\frac{1}{12}}(M_{1}M_{2})^{\frac{5}{12}}M_{1}^{\frac{1}{4}}\Big)^{\frac{2}{3}}
  \leq N^{\frac{8}{9}}Y^{\frac{1}{18}}(M_{1}M_{2})^{\frac{11}{18}},\\
J_{2}&\leq  \Big(N^{\frac{5}{6}}Y^{\frac{1}{12}}(M_{1}M_{2})^{\frac{5}{12}}M_{2}^{\frac{1}{4}}\Big)^{\frac{1}{2}}
\Big(N^{\frac{5}{6}}Y^{\frac{1}{12}}(M_{1}M_{2})^{\frac{5}{12}}M_{1}^{\frac{1}{4}}\Big)^{\frac{1}{2}}
 \leq N^{\frac{5}{6}}Y^{\frac{1}{12}}(M_{1}M_{2})^{\frac{13}{24}},\\
J_{3}&\leq  \Big(N(M_{1}M_{2})^{\frac{1}{2}}M_{2}^{\frac{1}{2}}\Big)^{\frac{1}{2}}\Big(N(M_{1}M_{2})^{\frac{1}{2}}M_{1}^{\frac{1}{2}}\Big)^{\frac{1}{2}}
\leq N(M_{1}M_{2})^{\frac{3}{4}},\\
J_{4}&\leq  \Big(N^{\frac{5}{6}}Y^{\frac{1}{12}}(M_{1}M_{2})^{\frac{5}{12}}M_{2}^{\frac{1}{4}}\Big)^{\frac{2}{3}}
\Big(N(M_{1}M_{2})^{\frac{1}{2}}M_{1}^{\frac{1}{2}}\Big)^{\frac{1}{3}}
\leq N^{\frac{8}{9}}Y^{\frac{1}{18}}(M_{1}M_{2})^{\frac{11}{18}}.
\end{aligned}
\end{equation*}
It now follows that
\begin{equation}\label{R*(m1,m2)}
\begin{split}
&R^{*}(M_{1},M_{2},N)Y^{-\varepsilon}\\
 &\ll 
N^{\frac{8}{9}}Y^{\frac{1}{18}}(M_{1}M_{2})^{\frac{11}{18}}+N(M_{1}M_{2})^{\frac{3}{4}}+N^{\frac{5}{6}}Y^{\frac{1}{12}}(M_{1}M_{2})^{\frac{13}{24}}
\\& 
\qquad +N^{\frac{5}{6}}Y^{-\frac{1}{6}}(M_{1}M_{2})^{\frac{7}{6}}
+N^{\frac{3}{4}}(M_{1}M_{2})
+N^{\frac{2}{3}}Y^{-\frac{1}{3}}(M_{1}M_{2})^{\frac{4}{3}}.
\end{split}
\end{equation}
Combining  \eqref{R(m1,m2)} with \eqref{R*(m1,m2)}, we  get (recalling $N\ll y$)
\begin{equation}\label{+R(m1,m2)}
\begin{split}
&R(M_{1},M_{2})Y^{-\varepsilon}\\
&\ll N^{\frac{2}{9}}Y^{\frac{7}{18}}(M_{1}M_{2})^{\frac{23}{18}}
+ N^{\frac{1}{3}}Y^{\frac{1}{3}}(M_{1}M_{2})^{\frac{17}{12}}
+ N^{\frac{1}{6}}Y^{\frac{5}{12}}(M_{1}M_{2})^{\frac{29}{24}}\\
&\qquad + N^{\frac{1}{6}}Y^{\frac{1}{6}}(M_{1}M_{2})^{\frac{11}{6}}
+ N^{\frac{1}{12}}Y^{\frac{1}{3}}(M_{1}M_{2})^{\frac{5}{3}}
+ Y^{\frac{2}{3}}(M_{1}M_{2})^{\frac{4}{3}}y^{-\frac{1}{3}}
+(M_{1}M_{2})^{2}\\
&\ll
Y^{\frac{7}{18}}y^{\frac{2}{9}}(M_{1}M_{2})^{\frac{23}{18}}
+ Y^{\frac{1}{3}}y^{\frac{1}{3}}(M_{1}M_{2})^{\frac{17}{12}}
+ Y^{\frac{5}{12}}y^{\frac{1}{6}}(M_{1}M_{2})^{\frac{29}{24}}\\
&\qquad + Y^{\frac{1}{6}}y^{\frac{1}{6}}(M_{1}M_{2})^{\frac{11}{6}}
+ Y^{\frac{1}{3}}y^{\frac{1}{12}}(M_{1}M_{2})^{\frac{5}{3}}
+ Y^{\frac{2}{3}}(M_{1}M_{2})^{\frac{4}{3}}y^{-\frac{1}{3}}
+(M_{1}M_{2})^{2}.
\end{split}
\end{equation}
By choosing a best $y$ with Lemma \ref{lemma 3} ~(recalling that $X_{l}=X/l$), we get that
\begin{equation}\label{R-4}
\begin{aligned}
R(M_{1},M_{2})Y^{-\varepsilon}
&\ll
Y^{\frac{1}{2}}(M_{1}M_{2})^{\frac{13}{10}}+
Y^{\frac{1}{2}}(M_{1}M_{2})^{\frac{11}{8}}+
Y^{\frac{1}{2}}(M_{1}M_{2})^{\frac{5}{4}}\\
&\qquad \qquad+
Y^{\frac{1}{3}}(M_{1}M_{2})^{\frac{5}{3}}+
Y^{\frac{2}{5}}(M_{1}M_{2})^{\frac{8}{5}}+
(M_{1}M_{2})^{2}\\
&\ll
Y^{\frac{1}{2}}{X_{l}}^{\frac{13}{10}}+
Y^{\frac{1}{2}}{X_{l}}^{\frac{11}{8}}+
Y^{\frac{1}{2}}{X_{l}}^{\frac{5}{4}}+
Y^{\frac{1}{3}}{X_{l}}^{\frac{5}{3}}+
Y^{\frac{2}{5}}{X_{l}}^{\frac{8}{5}}+
{X_{l}}^{2}.
\end{aligned}
\end{equation}

From \eqref{R_1+}-\eqref{R1-3} and \eqref{R-4}, we get
\begin{equation*}
\begin{aligned}
\mathfrak{R}_{1}^{\dag}Y^{-\varepsilon}&\ll
{X}^{\frac{13}{10}}Y^{\frac{1}{2}}+
{X}^{\frac{11}{8}}Y^{\frac{1}{2}}+
{X}^{\frac{5}{4}}Y^{\frac{1}{2}}+
{X}^{\frac{5}{3}}Y^{\frac{1}{3}}+
{X}^{\frac{8}{5}}Y^{\frac{2}{5}}+
{X}^{2}\\
&\ll
{X}^{\frac{11}{8}}Y^{\frac{1}{2}}+
{X}^{\frac{8}{5}}Y^{\frac{2}{5}}
\end{aligned}
\end{equation*}
by noting that $Y\geq X$.
This together with \eqref{2zong} and \eqref{R2} yields
\begin{equation*}
\mathfrak{R}Y^{-\varepsilon}\ll
{X}^{\frac{11}{8}}Y^{\frac{1}{2}}+
{X}^{\frac{8}{5}}Y^{\frac{2}{5}}+XY^{\frac{43}{96}}
\ll {X}^{\frac{11}{8}}Y^{\frac{1}{2}}+
{X}^{\frac{8}{5}}Y^{\frac{2}{5}}.
\end{equation*}
This completes the proof of Theorem 1.

\section{The proof of Theorem 2}
We begin with the first expression of $\mathfrak{R}$ in \eqref{2zong}
\begin{equation}\label{R_1+R_2}
\begin{aligned}
\mathfrak{R}&=\sum_{1\leq ml \leq X}ma_{\textit{K}}(m)\mu_{K}(l)P_{\textit{K}}\Big(\frac{Y}{m}\Big)\\
&=\sum_{1\leq m \leq X}ma_{\textit{K}}(m)\textit{M}_{\textit{K}}\Big(\frac{X}{m}\Big)P_{\textit{K}}\Big(\frac{Y}{m}\Big)\\
&=\mathfrak{R}_{1}+\mathfrak{R}_{2},
\end{aligned}
\end{equation}
where
\begin{equation*}
\begin{aligned}
&\mathfrak{R}_{1}=
\frac{Y^{1/3}}{\sqrt{3}\pi}\sum_{m\leq X}m^{2/3}a_{\textit{K}}(m)\textit{M}_{\textit{K}}\Big(\frac{X}{m}\Big)\sum_{n\leq y}
\frac{a_{\textit{K}}(n)}{n^{{2}/{3}}}\cos \Big(6\pi\sqrt[3]{\frac{nY}{m}}\Big),\\
&\mathfrak{R}_{2}=
\sum_{1\leq m \leq X}ma_{\textit{K}}(m)\textit{M}_{\textit{K}}\Big(\frac{X}{m}\Big)P_{2}\Big(\frac{Y}{m}\Big).
\end{aligned}
\end{equation*}

\vskip 3mm

A. \textit{Evaluation  of  $\int_{T}^{2T}\mathfrak{R}_{2}^{2}dY$}
\vskip 3mm
Suppose that $0<y<\big(\frac{T}{X}\big)^{1/3}$,
it is not hard to find that
\begin{equation*}
\begin{split}
\mathfrak{R}_{2}&\ll
\sum_{ m\sim M}ma_{\textit{K}}(m)\textit{M}_{\textit{K}}\Big(\frac{X}{m}\Big)P_{2}\Big(\frac{Y}{m}\Big)\log X\\
&\ll
X\sum_{ m\sim M}a_{\textit{K}}(m)P_{2}\Big(\frac{Y}{m}\Big)\log X
\end{split}
\end{equation*}
for some  $1\ll M\ll X$ and $\textit{M}_{\textit{K}}(t)\ll t$.
By Cauchy's inequality we get
\begin{equation*}
\begin{split}
\mathfrak{R}_{2}^{2}
&\ll
X^{2}\sum_{ m\sim M}a_{\textit{K}}(m)\sum_{ m\sim M}a_{\textit{K}}(m)P_{2}^{2}\Big(\frac{Y}{m}\Big)\log^{2}X\\
&\ll
X^{2}M\sum_{ m\sim M}a_{\textit{K}}(m)P_{2}^{2}\Big(\frac{Y}{m}\Big)\log^{2}X,
\end{split}
\end{equation*}
which together with $Xy^{3}\ll T$ imply that
\begin{equation}\label{r2r2}
\begin{split}
\int_{T}^{2T}\mathfrak{R}_{2}^{2}dY
&\ll
X^{2}M\sum_{ m\sim M}a_{\textit{K}}(m)\log^{2}X\int_{T}^{2T}P_{2}^{2}\Big(\frac{Y}{m}\Big)dY\\
&\ll
X^{2}M\sum_{ m\sim M}a_{\textit{K}}(m)m\log^{2}X\int_{T}^{2T}P_{2}^{2}\Big(\frac{Y}{m}\Big)d\Big(\frac{Y}{m}\Big)\\
&\ll
X^{2}M\sum_{ m\sim M}a_{\textit{K}}(m)m\Big(\frac{T}{m}\Big)^{\frac{5}{3}+\varepsilon}y^{-\frac{1}{3}}\log^{2}X\\
&\ll
X^{2}M^{\frac{4}{3}}T^{\frac{5}{3}+\varepsilon}y^{-\frac{1}{3}}\\
&\ll
X^{\frac{10}{3}}T^{\frac{5}{3}+\varepsilon}y^{-\frac{1}{3}}.
\end{split}
\end{equation}

\vskip 3mm

B. \textit{Evaluation  of  $\int_{T}^{2T}\mathfrak{R}_{1}^{2}dY$}

\vskip 3mm

Noticing that
\begin{equation*}
\begin{aligned}
\mathfrak{R}_{1}^{2}=\frac{Y^{\frac{2}{3}}}{3\pi^{2}}&
\sum_{1\leq m_{1},m_{2}\leq X}(m_{1}m_{2})^{\frac{2}{3}}a_{\textit{K}}(m_{1})a_{\textit{K}}(m_{2})\textit{M}_{\textit{K}}\Big(\frac{X}{m_{1}}\Big)\textit{M}_{\textit{K}}\Big(\frac{X}{m_{2}}\Big)\\
&\times \sum_{n_{1},n_{2}\leq y}\frac{a_{\textit{K}}(n_{1})}{n_{1}^{{2}/{3}}}\frac{a_{\textit{K}}(n_{2})}{n_{2}^{{2}/{3}}}\cos \Big(6\pi\sqrt[3]{\frac{n_{1}Y}{m_{1}}}\Big)\cos \Big(6\pi\sqrt[3]{\frac{n_{2}Y}{m_{2}}}\Big)
\end{aligned}
\end{equation*}
and using the elementary formula $
\cos \alpha \cos \beta=\frac{1}{2}\big(\cos (\alpha-\beta)+\cos (\alpha+\beta)\big)
$
give
\begin{equation}\label{Q1+Q2+Q3}
\mathfrak{R}_{1}^{2}=Q_{1}(Y)+Q_{2}(Y)+Q_{3}(Y),
\end{equation}
where
\begin{equation*}
\begin{aligned}
Q_{1}(Y)
&  :=\frac{Y^{\frac{2}{3}}}{6\pi^{2}}\sum_{\substack{m_{1},m_{2}\leq X;n_{1},n_{2}\leq y\\n_{1}m_{2}=n_{2}m_{1}}}(m_{1}m_{2})^{\frac{2}{3}}a_{\textit{K}}(m_{1})a_{\textit{K}}(m_{2})\\
&\qquad       \qquad  \qquad \qquad \times\textit{M}_{\textit{K}}\Big(\frac{X}{m_{1}}\Big)\textit{M}_{\textit{K}}\Big(\frac{X}{m_{2}}\Big)
\frac{a_{\textit{K}}(n_{1})}{n_{1}^{{2}/{3}}}\frac{a_{\textit{K}}(n_{2})}{n_{2}^{{2}/{3}}},
\\ \noalign{\vskip 1mm}
Q_{2}(Y)
&  :=\frac{Y^{\frac{2}{3}}}{6\pi^{2}}\sum_{\substack{m_{1},m_{2}\leq X;n_{1},n_{2}\leq y\\n_{1}m_{2}\neq n_{2}m_{1}}}(m_{1}m_{2})^{\frac{2}{3}}a_{\textit{K}}(m_{1})a_{\textit{K}}(m_{2})\textit{M}_{\textit{K}}\Big(\frac{X}{m_{1}}\Big)\textit{M}_{\textit{K}}\Big(\frac{X}{m_{2}}\Big)
\\ \noalign{\vskip 1mm}
&\qquad       \qquad  \qquad  \qquad \times\frac{a_{\textit{K}}(n_{1})}{n_{1}^{{2}/{3}}}\frac{a_{\textit{K}}(n_{2})}{n_{2}^{{2}/{3}}}\cos \Big(6\pi\sqrt[3]{Y}\Big(\sqrt[3]{\frac{n_{1}}{m_{1}}}-\sqrt[3]{\frac{n_{2}}{m_{2}}}\Big)\Big),
\\ \noalign{\vskip 1mm}
Q_{3}(Y)
&  :=\frac{Y^{\frac{2}{3}}}{6\pi^{2}}\sum_{\substack{m_{1},m_{2}\leq X\\n_{1},n_{2}\leq y}}(m_{1}m_{2})^{\frac{2}{3}}a_{\textit{K}}(m_{1})a_{\textit{K}}(m_{2})\textit{M}_{\textit{K}}\Big(\frac{X}{m_{1}}\Big)\textit{M}_{\textit{K}}\Big(\frac{X}{m_{2}}\Big)
\\
& \qquad  \qquad     \qquad \qquad
\times \frac{a_{\textit{K}}(n_{1})}{n_{1}^{{2}/{3}}}\frac{a_{\textit{K}}(n_{2})}{n_{2}^{{2}/{3}}}\cos \Big(6\pi\sqrt[3]{Y}\Big(\sqrt[3]{\frac{n_{1}}{m_{1}}}+\sqrt[3]{\frac{n_{2}}{m_{2}}}\Big)\Big).
\end{aligned}
\end{equation*}

Firstly, we consider $Q_{3}(Y)$. By using the first derivative test, \eqref{M(x)} and the elementary formula $a+b\geq 2\sqrt{ab}$ ($a>0, b>0$), we get
\begin{equation}\label{3-1}
\begin{aligned}
\int_{T}^{2T}Q_{3}(Y)dY
&\ll T^{\frac{4}{3}}\sum_{\substack{m_{1},m_{2}\leq X\\n_{1},n_{2}\leq y}}(m_{1}m_{2})^{\frac{2}{3}}a_{\textit{K}}(m_{1})a_{\textit{K}}(m_{2})\Big|\textit{M}_{\textit{K}}\Big(\frac{X}{m_{1}}\Big)\textit{M}_{\textit{K}}\Big(\frac{X}{m_{2}}\Big)\Big|\\
& \qquad  \qquad     \qquad
\times \frac{a_{\textit{K}}(n_{1})}{n_{1}^{{2}/{3}}}\frac{a_{\textit{K}}(n_{2})}{n_{2}^{{2}/{3}}}\times \frac{1}{\sqrt[3]{\frac{n_{1}}{m_{1}}}+\sqrt[3]{\frac{n_{2}}{m_{2}}}}
\\
&\ll X^{2}T^{\frac{4}{3}}\sum_{m_{1},m_{2}\leq X}\frac{a_{\textit{K}}(m_{1})a_{\textit{K}}(m_{2})}{(m_{1}m_{2})^{{1}/{6}}}\sum_{n_{1},n_{2}\leq y}\frac{a_{\textit{K}}(m)}{n_{1}^{{5}/{6}}}\frac{a_{\textit{K}}(m)}{n_{2}^{{5}/{6}}}
\\&
\ll X^{\frac{11}{3}}T^{\frac{4}{3}}y^{\frac{1}{3}},
\end{aligned}
\end{equation}
where in the last step we used \eqref{dd} and partial summation.

Secondly, we consider $Q_{2}(Y)$. By the first derivative test and \eqref{M(x)} again we get with the help of Lemma \ref{lemma 7} that
\begin{equation}\label{3-2}
\begin{aligned}
&\int_{T}^{2T}Q_{2}(Y)dY\\
&\ll T^{\frac{4}{3}}\sum_{\substack{m_{1},m_{2}\leq X;n_{1},n_{2}\leq y\\n_{1}m_{2}\neq n_{2}m_{1}}}(m_{1}m_{2})^{\frac{2}{3}}a_{\textit{K}}(m_{1})a_{\textit{K}}(m_{2})\Big|\textit{M}_{\textit{K}}\Big(\frac{X}{m_{1}}\Big)\textit{M}_{\textit{K}}\Big(\frac{X}{m_{2}}\Big)\Big|\\
& \qquad  \qquad     \qquad
\times \frac{a_{\textit{K}}(n_{1})}{n_{1}^{{2}/{3}}}\frac{a_{\textit{K}}(n_{2})}{n_{2}^{{2}/{3}}}\times \frac{1}{\Big |\sqrt[3]{\frac{n_{1}}{m_{1}}}-\sqrt[3]{\frac{n_{2}}{m_{2}}}\Big|}
\\
&\ll X^{2}T^{\frac{4}{3}} \sum_{\substack{m_{1},m_{2}\leq X;n_{1},n_{2}\leq y\\n_{1}m_{2}\neq n_{2}m_{1}}}
\frac{a_{\textit{K}}(m_{1})a_{\textit{K}}(m_{2})a_{\textit{K}}(n_{1})a_{\textit{K}}(n_{2})}{(n_{1}n_{2})^{{2}/{3}}|\sqrt[3]{n_{1}m_{2}}-\sqrt[3]{n_{2}m_{1}}|}
\\&
\ll X^{\frac{10}{3}}T^{\frac{4}{3}}\sum_{\substack{m_{1},m_{2}\leq X;n_{1},n_{2}\leq y\\n_{1}m_{2}\neq n_{2}m_{1}}}
\frac{a_{\textit{K}}(m_{1})a_{\textit{K}}(m_{2})a_{\textit{K}}(n_{1})a_{\textit{K}}(n_{2})}{(m_{1}m_{2})^{{2}/{3}}(n_{1}n_{2})^{{2}/{3}}|\sqrt[3]{n_{1}m_{2}}-\sqrt[3]{n_{2}m_{1}}|}
\\&
\ll X^{\frac{10}{3}}T^{\frac{4}{3}}\sum_{\substack{l_{1},l_{2}\leq Xy\\l_{1}\neq l_{2}}}\frac{\tau_{4}^{2}(l_{1})\tau_{4}^{2}(l_{2})}{l_{1}^{2/3}l_{2}^{2/3}|\sqrt[3]{l_{1}}-\sqrt[3]{l_{2}}|}\\
&
\ll T^{\frac{4}{3}}X^{\frac{10}{3}}(Xy)^{\frac{1}{3}+\varepsilon}\\
&\ll X^{\frac{11}{3}}T^{\frac{4}{3}+\varepsilon}y^{\frac{1}{3}},
\end{aligned}
\end{equation}
where we used the estimate $a_{\textit{K}}(m)a_{\textit{K}}(n)\leq \tau^{2}(m)\tau^{2}(n)\leq \tau_{4}^{2}(mn)$.

Finally, we consider $Q_{1}(Y)$.
Let $m=(m_{1},m_{2})$. Write $m_{1}=mm_{1}^{*},~ m_{2}=mm_{2}^{*}$  such that $(m_{1}^{*}, ~m_{2}^{*})=1$. If $n_{1}m_{2}=n_{2}m_{1}$, we immediately get that
$n_{1}=nm_{1}^{*}, ~n_{2}=nm_{2}^{*}$  for some positive integer $n$. It follows that
\begin{equation}\label{Q1(Y)}
\begin{aligned}
Q_{1}(Y)&=\frac{Y^{\frac{2}{3}}}{6\pi^{2}}\sum_{\substack{mm_{1},mm_{2}\leq X\\gcd(m_{1},m_{2})=1}}m^{4/3}a_{\textit{K}}(mm_{1})a_{\textit{K}}(mm_{2})\textit{M}_{\textit{K}}\Big(
\frac{X}{mm_{1}}\Big)\textit{M}_{\textit{K}}\Big({\frac{X}{mm_{2}}}\Big)
\\& \qquad   \qquad \qquad
\times \sum_{n\leq \min(\frac{y}{m_{1}},\frac{y}{m_{2}})}\frac{a_{\textit{K}}(nm_{1})a_{\textit{K}}(nm_{2})}{n^{4/3}}\\
&=c(X)Y^{\frac{2}{3}}+E(Y),
\end{aligned}
\end{equation}
where
\begin{equation}
\begin{aligned}\label{cx}
&c(X)=\frac{1}{6\pi^{2}}\sum_{\substack{mm_{1},mm_{2}\leq X\\gcd(m_{1},m_{2})=1}}m^{4/3}a_{\textit{K}}(mm_{1})a_{\textit{K}}(mm_{2})\textit{M}_{\textit{K}}\Big(\frac{X}{mm_{1}}\Big)\textit{M}_{\textit{K}}\Big(\frac{X}{mm_{2}}\Big)
\\& \qquad   \qquad \qquad
\times \sum_{n=1}^{\infty}\frac{a_{\textit{K}}(nm_{1})a_{\textit{K}}(nm_{2})}{n^{4/3}},\\
&
E(Y)=\frac{Y^{\frac{2}{3}}}{6\pi^{2}}\sum_{\substack{mm_{1},mm_{2}\leq X\\gcd(m_{1},m_{2})=1}}m^{4/3}a_{\textit{K}}(mm_{1})a_{\textit{K}}(mm_{2})\textit{M}_{\textit{K}}\Big(\frac{X}{mm_{1}}\Big)\textit{M}_{\textit{K}}\Big(\frac{X}{mm_{2}}\Big)
\\& \qquad   \qquad \qquad
\times \sum_{n> \min(\frac{y}{m_{1}},\frac{y}{m_{2}})}\frac{a_{\textit{K}}(nm_{1})a_{\textit{K}}(nm_{2})}{n^{4/3}}.
\end{aligned}
\end{equation}
Noting that $a_{\textit{K}}(mn)\leq \tau^{2}(mn)\leq \tau^{2}(m)\tau^{2}(n)$, we get that
\begin{equation}\label{E(Y)}
\begin{aligned}
E(Y)&\ll Y^{\frac{2}{3}}\sum_{\substack{mm_{1},mm_{2}\leq X\\gcd(m_{1},m_{2})=1}}m^{4/3}a_{\textit{K}}(mm_{1})a_{\textit{K}}(mm_{2})\textit{M}_{\textit{K}}\Big(\frac{X}{mm_{1}}\Big)\textit{M}_{\textit{K}}\Big(\frac{X}{mm_{2}}\Big)
\\& \qquad   \qquad \qquad
\times \sum_{n> \min(\frac{y}{m_{1}},\frac{y}{m_{2}})}\frac{a_{\textit{K}}(nm_{1})a_{\textit{K}}(nm_{2})}{n^{4/3}}\\
&\ll X^{2}Y^{\frac{2}{3}}\sum_{m\leq X}\frac{\tau^{4}(m)}{m^{2/3}}\sum_{\substack{m_{1}\leq \frac{X}{m},m_{2}\leq \frac{X}{m}\\gcd(m_{1},m_{2})=1}}\frac{\tau^{4}(m_{1})\tau^{4}(m_{2})}{m_{1}m_{2}}\sum_{n>\min(\frac{y}{m_{1}},\frac{y}{m_{2}})}\frac{\tau^{4}(n)}{n^{4/3}}
\\&
\ll X^{2}Y^{\frac{2}{3}}\sum_{m\leq X}\frac{\tau^{4}(m)}{m^{2/3}}\sum_{m_{1}\leq m_{2}\leq \frac{X}{m}}\frac{\tau^{4}(m_{1})\tau^{4}(m_{2})}{m_{1}m_{2}}
\times\Big(\frac{m_{2}}{y}\Big)^{1/3-\varepsilon}\\
&
\ll X^{2}Y^{\frac{2}{3}}{y}^{\varepsilon-\frac{1}{3}}\sum_{m\leq X}\frac{\tau^{4}(m)}{m^{2/3}}\sum_{m_{2}\leq \frac{X}{m}}\frac{\tau^{4}(m_{2})}{m_{2}^{2/3+\varepsilon}}\sum_{m_{1}\leq m_{2}}\frac{\tau^{4}(m_{1})}{m_{1}}\\
&
\ll X^{\frac{7}{3}}T^{\frac{2}{3}+\varepsilon}y^{-\frac{1}{3}}.
\end{aligned}
\end{equation}
This together with \eqref{Q1(Y)} yields
\begin{equation}\label{3-3}
\int_{T}^{2T}Q_{1}(Y)dY=c(X)\int_{T}^{2T}Y^{\frac{2}{3}}dY+O(X^{\frac{7}{3}}T^{\frac{5}{3}+\varepsilon}y^{-\frac{1}{3}}).
\end{equation}
Similar to \eqref{E(Y)}, we obtain the estimate
\begin{equation}\label{c(x)}
c(X)\ll X^{\frac{7}{3}+\varepsilon}.
\end{equation}

From \eqref{Q1+Q2+Q3}-\eqref{3-2} and \eqref{3-3}, we get
\begin{equation}\label{r1r1}
\int_{T}^{2T}\mathfrak{R}_{1}^{2}dY=c(X)\int_{T}^{2T}Y^{\frac{2}{3}}dY+O(X^{\frac{7}{3}}T^{\frac{5}{3}+\varepsilon}y^{-\frac{1}{3}}+
X^{\frac{11}{3}}T^{\frac{4}{3}+\varepsilon}y^{\frac{1}{3}}).
\end{equation}

\vskip 3mm

C. \textit{Evaluation   of  $\int_{T}^{2T}\mathfrak{R}^{2}dY$}

\vskip 3mm

From \eqref{r2r2}, \eqref{c(x)}, \eqref{r1r1} and  Cauchy's inequality, we get
\begin{equation}\label{r1r2}
\begin{aligned}
\int_{T}^{2T}\mathfrak{R}_{1}\mathfrak{R}_{2}dY\ll
X^{\frac{17}{6}}T^{\frac{5}{3}+\varepsilon}y^{-\frac{1}{6}}
+X^{\frac{7}{2}}T^{\frac{3}{2}+\varepsilon}.
\end{aligned}
\end{equation}
Combining \eqref{R_1+R_2}, \eqref{r2r2} and \eqref{r1r1}, we finally  get
\begin{equation}\label{iiiiiiiiii}
\begin{split}
\int_{T}^{2T}\mathfrak{R}^{2}dY
&=
c(X)\int_{T}^{2T}Y^{\frac{2}{3}}dY+
O(X^{\frac{11}{3}}T^{\frac{4}{3}+\varepsilon}y^{\frac{1}{3}}+X^{\frac{10}{3}}T^{\frac{5}{3}+\varepsilon}y^{-\frac{1}{3}}\\
&\qquad \qquad \qquad \qquad\qquad \qquad+X^{\frac{17}{6}}T^{\frac{5}{3}+\varepsilon}y^{-\frac{1}{6}}+X^{\frac{7}{2}}T^{\frac{3}{2}+\varepsilon}).
\end{split}
\end{equation}

By choosing a best $y\in (1,(T/X)^{1/3})$ via Lemma \ref{lemma 5}, we get
$$
\int_{T}^{2T}|\mathfrak{R}_{\textit{K}}(X,Y)|^{2}dY=
c(X)\int_{T}^{2T}Y^{\frac{2}{3}}dY+
O(X^{\frac{31}{9}}T^{\frac{14}{9}+\varepsilon}+X^{\frac{26}{9}}T^{\frac{29}{18}+\varepsilon})
,
$$
where $c(X)$ is defined by \eqref{cx}.
This completes the proof of Theorem 2.

\end{document}